\documentclass[12pt]{amsart}
\usepackage{amsfonts,amssymb,amscd,amsmath,enumerate,verbatim,amsthm}
\usepackage[latin1]{inputenc}
\usepackage[subrefformat=parens]{subcaption}
\usepackage{amscd}
\usepackage{latexsym}
\usepackage{tikz,xcolor}

\usepackage{graphicx}

\usepackage{mathptmx}

\def\NZQ{\mathbb}               

\def\ZZ{{\NZQ Z}}
\def\RR{{\NZQ R}}

\def\ab{{\mathbf a}}

\def\eb{{\mathbf e}}
\def\fb{{\mathbf f}}

\def\xb{{\mathbf x}}

\def\Ib{{\mathbf I}}

\def\opn#1#2{\def#1{\operatorname{#2}}} 
\opn\ini{in} \opn\sgn{sgn} \opn\Gr{Gr} \opn\Im{Im}

\def\Ac{{\mathcal A}}

\def\Rc{{\mathcal R}}

\def\Ic{{\mathcal I}}

\def\Gc{{\mathcal G}}
\def\Fc{{\mathcal F}}

\def\Rc{{\mathcal R}}

\def\kk{{\Bbbk}}
\newtheorem{thm}{Theorem}[section]

\newtheorem{cor}[thm]{Corollary}
\newtheorem{prop}[thm]{Proposition}

\newtheorem{prob}[thm]{Problem}
\newtheorem{q}[thm]{Question}

\theoremstyle{definition}
\newtheorem{ex}[thm]{Example}

\theoremstyle{remark}
\newtheorem{rem}[thm]{Remark}

\begin{document}

\title{Binomial edge rings of complete bipartite graphs}

\author{Akihiro Higashitani}

\address{Akihiro Higashitani,
Graduate School of Information Science and Technology,
Osaka University, Japan}
\email{higashitani@ist.osaka-u.ac.jp}


\begin{abstract}
We introduce a new class of algebras arising from graphs, called binomial edge rings. 
Given a graph $G$ on $d$ vertices with $n$ edges, the binomial edge ring of $G$ is defined to be the subalgebra of the polynomial ring with $2d$ variables 
generated by the binomials which correspond to $n$ edges. 
In this paper, we calculate a SAGBI basis for this algebra and obtain an initial algebra associated with this SAGBI basis in the case of complete bipartite graphs. 
It turns out that such an initial algebra is isomorphic to the Hibi ring of a certain poset. 
Similar phenomenon also occurs in the context of Pl\"{u}cker algebras, so the framework of binomial edge rings can be interpreted as a kind of its generalization. 
\end{abstract}

\maketitle

\section{Introduction}

Throughout this paper, let $\kk$ be a field. 
All graphs appearing in this paper are always assumed to be finite, simple and connected. 

\subsection{Ideals and algebras arising from graphs}

One of the main topics in the theory of combinatorial commutative algebra 
is the study of ideals/algebras arising from combinatorial objects, such as graphs, posets, simplicial complexes, matroids, and so on. 
In particular, there are several kinds of ideals/algebras arising from graphs. 
We briefly introduce some of them below.

Let $G$ be a graph on the vertex $V(G)=[d]:=\{1,\ldots,d\}$ with the edge set $E(G)$. 
Let \[R=\kk[x_1,\ldots,x_d] \;\text{ and }\; S=\kk[x_1,\ldots,x_d,y_1,\ldots,y_d]\] 
be the polynomial rings in $d$ and $2d$ variables over $\kk$, respectively.  
\begin{itemize}
\item The \textit{edge ideal} of $G$ is defined to be the ideal \[(x_ix_j : \{i,j\} \in E(G)) \subset R\] 
generated by quadratic monomials corresponding to the edges of $G$. 
As far as the author knows, the study of edge ideals seems initiated in \cite{simis-vasconcelos-villarreal-1994}. 
Nowadays, there are a huge number of studies on edge ideals of graphs. 
For an introduction to edge ideals and some fundamental results on them, see, e.g., \cite[Section 9]{herzog-hibi-book} and \cite[Section 7]{villarreal}. 
\item The \textit{edge ring} of $G$ is defined to be the subalgebra \[\kk[x_ix_j : \{i,j\} \in E(G)] \subset R\] 
generated by quadratic monomials corresponding to the edges of $G$.  
The intensive investigation of this $\kk$-algebra has begun after the works \cite{ohsugi-hibi-1998} and \cite{simis-vasconcelos-villarreal-1998}. 
There are also numerous studies on edge rings and their defining ideals, called the \textit{toric ideals of $G$}. 
See, e.g., \cite[Section 5]{herzog-hibi-ohsugi} and \cite[Sections 10 and 11]{villarreal} for their introduction. 
\item The \textit{binomial edge ideal} of $G$ is defined to be the ideal \[(x_iy_j-x_jy_i : \{i,j\} \in E(G)) \subset S\] 
generated by quadratic binomials corresponding to the edges of $G$. 
This ideal has been independently introduced by \cite{HHHKR-2010} and \cite{ohtani-2012}. 
There are also a large number of studies on binomial edge ideals. See, e.g., \cite[Section 7]{herzog-hibi-ohsugi} for their introduction.
\end{itemize}

Following these trends, in this paper, we introduce a new class of $\kk$-algebras arising from $G$. Let  
$$\Rc(G):=\kk[x_iy_j-x_jy_i : \{i,j\} \in E(G)] \subset S.$$
We call this subalgebra of $S$ the \textit{binomial edge ring} of $G$. 

Given $1 \leq i<j \leq d$, we employ the following notation: 
\[f_{ij}:=x_iy_j-x_jy_i \in S.\]
Namely, $\Rc(G)$ is the subalgebra of $S$ generated by $\{f_{ij} : \{i,j\} \in E(G)\}$.

\subsection{Binomial edge rings of graphs and Pl\"{u}cker algebras}
To explain other motivation of introducing binomial edge rings of graphs, let us briefly recall the fundamental facts on Pl\"{u}cker algebras. 

Let $1 \leq k< d$ and let $X=(x_{ij})_{1 \leq i \leq k, 1 \leq j \leq d}$ be a $k \times d$-matrix of indeterminates. 
Given $I \in \Ib_{k,d}:=\{I \subset [d] : |I|=k\}$, let $X_I$ denote the $k \times k$-submatrix of $X$ whose columns are indexed by $I$. 
We define the subalgebra of the polynomial ring $\kk[x_{ij} : 1 \leq i \leq k, 1 \leq j \leq d]$ in $kd$ variables as follows: 
\[\Ac_{k,d}:=\kk[\det(X_I) : I \in \Ib_{k,d}] \subset \kk[x_{ij} : 1 \leq i \leq k, 1 \leq j \leq d].\]
The algebra $\Ac_{k,d}$ is called the \textit{Pl\"{u}cker algebra}. 
This is well known as a homogeneous coordinate ring of the Grassmannian $\Gr_\kk(k,d)$, the set of $k$-subspaces of $\kk^d$. 
Here, we notice that $\Ac_{2,d}$ coincides with the binomial edge ring of the complete graph $K_d$. 
Hence, we can regard $\Rc(G)$ as a kind of generalization of the Pl\"{u}cker algebra. 

For the introduction to SAGBI basis and initial algebra, see Subsection~\ref{sec:SAGBI}. 
It is known that $\{\det(X_I) : I \in \Ib_{k,d}\}$ forms a SAGBI basis for $\Ac_{k,d}$ with respect to a certain monomial order. 
Moreover, the associated initial algebra is isomorphic to a Hibi ring (see Subsection~\ref{sec:Hibi}) of a certain poset. 
See, e.g., \cite{miller-sturmfels} and \cite{sturmfels}.

On the other hand, whether $\{\det(X_I) : I \in \Ib_{k,d}\}$ forms a SAGBI basis or not strongly depends on a choice of monomial orders. 
For example, computational experiments on this are performed in \cite{bruns-conca}. 

Therefore, in the context of binomial edge rings, it is natural to think of a SAGBI basis of $\Rc(G)$ and to discuss the structure of the associated initial algebra.

\subsection{Main Results}
Our main interest is the binomial edge ring of the complete graph $K_{a,b}$ with $2 \leq a \leq b$ and a SAGBI basis of $\Rc(K_{a,b})$. 
First, we give a SAGBI basis for $\Rc(K_{a,b})$, which is the first main result of this paper.  
\begin{thm}\label{main1} Let $2 \leq a \leq b$ and let \begin{align*}
\Gc_{a,b}=&\{f_{ij} : 1 \leq i \leq a, \, 1 \leq j \leq b\} \\
\cup &\{f_{ij'}f_{i'j} - f_{ij}f_{i'j'} : 1 \leq i<i' \leq a, \, 1 \leq j'<j \leq b\}. 
\end{align*}
Then $\Gc_{a,b}$ forms a SAGBI basis of $\Rc(K_{a,b})$ with respect to the monomial order $<$ described in \eqref{order}. 
\end{thm}
See Subsection~\ref{sec:label} for the detail of the notation. 

Once a SAGBI basis is given, we obtain an initial algebra of $\Rc(G)$ which encodes its several ring-theoretic properties. 
Let $\Pi_{a,b}$ be the poset depicted in Figure~\ref{piab}: 
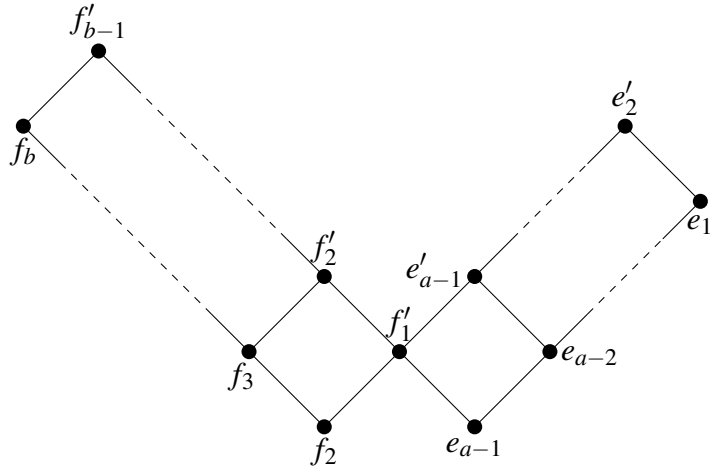
\begin{figure}[h]
\begin{tikzpicture}
    \draw (0,1) -- (1,0);
    \draw (1,2) -- (2,1);
    \draw (3,4) -- (4,3);
    \draw (0,1) -- (-1,0);
    \draw (-1,2) -- (-2,1);
    \draw (-4,5) -- (-5,4);

    \draw (1,0) -- (2.5,1.5);
    \draw (0,1) -- (1.5,2.5);
    \draw (-1,0) -- (-2.5,1.5);
    \draw (0,1) -- (-1.5,2.5);
    \draw (3.5,2.5) -- (4,3);
    \draw (2.5,3.5) -- (3,4);
    \draw (-4,5) -- (-3.5,4.5);
    \draw (-5,4) -- (-4.5,3.5);

    \draw[dashed] (2.5,1.5) -- (3.5,2.5);
    \draw[dashed] (1.5,2.5) -- (2.5,3.5);
    \draw[dashed] (-1.5,2.5) -- (-3.5,4.5);
    \draw[dashed] (-2.5,1.5) -- (-4.5,3.5);

    \fill [black, inner sep=2pt] (0,1) circle(0.1) node at (0,1) [above=2pt] {$f_1'$};
    \fill [black, inner sep=2pt] (1,0) circle(0.1) node at (1,0) [below=2pt] {$e_{a-1}$};
    \fill [black, inner sep=2pt] (1,2) circle(0.1) node at (1,2) [above=2pt, left=2pt] {$e_{a-1}'$};
    \fill [black, inner sep=2pt] (2,1) circle(0.1) node at (2,1) [below=2pt, right=2pt] {$e_{a-2}$};
    \fill [black, inner sep=2pt] (3,4) circle(0.1) node at (3,4) [above=2pt] {$e_2'$};
    \fill [black, inner sep=2pt] (4,3) circle(0.1) node at (4,3) [below=3pt] {$e_1$};
    \fill [black, inner sep=2pt] (-1,0) circle(0.1) node at (-1,0) [below=2pt] {$f_2$};
    \fill [black, inner sep=2pt] (-1,2) circle(0.1) node at (-1,2) [above=2pt] {$f_2'$};
    \fill [black, inner sep=2pt] (-2,1) circle(0.1) node at (-2.1,1) [below=2pt] {$f_3$};
    \fill [black, inner sep=2pt] (-4,5) circle(0.1) node at (-4,5) [above=2pt] {$f_{b-1}'$};
    \fill [black, inner sep=2pt] (-5,4) circle(0.1) node at (-5,4) [below=2pt] {$f_b$};
\end{tikzpicture}
\caption{The associated poset $\Pi_{a,b}$}\label{piab}
\end{figure}

The following theorem and its corollary are the second main result of this paper. 
\begin{thm}\label{main2}
Let $2 \leq a \leq b$. 
The initial algebra of $\Rc(K_{a,b})$ with respect to $<$ described in \eqref{order} is isomorphic to the Hibi ring of $\Pi_{a,b}$.  
\end{thm}

\begin{cor}\label{kei} Let $2 \leq a \leq b$. Then we have the following: 
\begin{itemize}
\item $\dim (\Rc(K_{a,b}))=2(a+b-2)$; 
\item $\Rc(K_{a,b})$ is a Cohen--Macaulay domain; 
\item $\Rc(K_{a,b})$ is Gorenstein if and only if $a=2$ or $a=b$. 
\end{itemize}
\end{cor}

\begin{rem}
The graph $K_{1,b}$ is, so-called, a \textit{star graph}, which is a kind of trees. 
If $T$ is a tree, then we see that $\Rc(T)$ is isomorphic to a polynomial ring in $|E(T)|$ variables. 
For more details, see Example~\ref{ex:tree} (i). 
\end{rem}

\subsection{Structure of the paper}
A brief structure of this paper is as follows. 
In Section~\ref{sec;pre}, we prepare two central materials of this paper, SAGBI basis and Hibi rings. 
We also give a proof of Corollary~\ref{kei} here. 
In Section~\ref{sec:proofs}, we prove Theorems~\ref{main1} and \ref{main2} simultaneously. 
In Section~\ref{sec:q}, we discuss $\Rc(G)$ in the case of general graphs $G$ and suggest some further questions. 

\subsection*{Acknowledgements}
The author is partially supported by JSPS KAKENHI Grant Number JP24K00521 and JP21KK0043. 

\bigskip

\section{Preliminaries}\label{sec;pre}

The goal of this section is to introduce the notions of SAGBI basis and Hibi rings. Those play the central role in this paper. 

\subsection{SAGBI basis of a subalgebra}\label{sec:SAGBI}

First, we recall SAGBI basis from \cite{sturmfels}. 
(Remark that \textit{canonical basis} is used in \cite{sturmfels} instead of SAGBI basis, but those are the same things.) 
Fix a set of polynomials $\Fc=\{f_1,\ldots,f_n\}$ of $R$ and a monomial order $<$ on $R$. 
Consider the subalgebra $\kk[\Fc]$ of $R$ generated by the polynomials $f_1,\ldots,f_n$. 
The \textit{initial algebra} $\ini_<(\kk[\Fc])$ is the algebra generated by the initial monomials $\ini_<(g)$ for $g \in \kk[\Fc]$. 
Note that $\kk[\Fc]$ is not necessarily a finitely generated $\kk$-algebra even if $\kk[\Fc]$ is finitely generated (see \cite{Robbiano-Sweedler}). 
We say that $\Fc$ is a \textit{SAGBI basis} of $\kk[\Fc]$ with respect to $<$ 
if the initial algebra $\ini_<(\kk[\Fc])$ coincides with the subalgebra $\kk[\ini_<(\Fc)]$ generated by the initial monomials $\ini_<(f_1),\ldots,\ini_<(f_n)$. 
The word ``SAGBI'' is introduced by Robbiano and Sweedler \cite{Robbiano-Sweedler} and means ``Subalgebra Analog to Gr{\"o}bner Bases for Ideals''.

We discuss when $\Fc$ becomes a SAGBI basis of $\kk[\Fc]$ in terms of the associated toric ideal. 
For $i=1,\ldots,n$, let $\ab_i=(a_{1i},\ldots,a_{di}) \in \ZZ_{\geq 0}^d$ be such that $\ini_<(f_i)=\xb^{\ab_i}=x_1^{a_{1i}}\cdots x_d^{a_{di}}$ and 
let $A$ be the $d \times n$-matrix whose columns consist of $\ab_1,\ldots,\ab_n$. 
Let $I$ be the kernel of the following surjective ring homomorphism: 
\[ \kk[z_1,\ldots,z_n] \rightarrow \kk[\Fc]; \;\; z_i \mapsto f_i. \]
Similarly, let $I_A$ be the kernel of the surjective ring homomorphism: 
\[ \kk[z_1,\ldots,z_n] \rightarrow \kk[\xb^{\ab_1},\ldots,\xb^{\ab_n}]; \;\; z_i \mapsto \xb^{\ab_i}. \]

Take any weight vector $w \in \RR^d$ representing the monomial order $<$ for $\Fc$. Namely, $\ini_<(f_i)=\ini_w(f_i)$ holds for each $i$. 
Then we can use $A^T w \in \RR^n$ as a weight vector for $\kk[x_1,\ldots,x_n]$, where $A^T$ stands for the transpose of $A$.   
Remark that the initial ideal $\ini_{A^Tw}(I)$ is usually not a monomial ideal. 
We expect this is a binomial ideal as the following theorem claims.  
\begin{thm}[{\cite[Lemma 11.3 and Theorem 11.4]{sturmfels}}]\label{sagbi}
For any finite set $\Fc \subset R$, the inclusion $\ini_{A^Tw}(I) \subseteq I_A$ holds. 
Moreover, the equality of this inclusion holds if and only if $\Fc$ forms a SAGBI basis of $\kk[\Fc]$ with respect to $<$. 
\end{thm}

Thanks to the fundamental theory on SAGBI basis, we can prove Corollary~\ref{kei} by using Theorem~\ref{main2}. 
\begin{proof}[Proof of Corollary~\ref{kei}]
In general, for a finitely generated graded $\kk$-algebra $T$ and a fixed monomial order $<$, 
we know that the Hilbert series of $T$ coincides with that of its initial algebra $\ini_<(T)$ (see, e.g., \cite{Robbiano-Sweedler}). 
In particular, $\dim T=\dim (\ini_<(T))$. Since the Krull dimension of the Hibi ring of a poset $\Pi$ is equal to $|\Pi|+1$ (see Subsection~\ref{sec:Hibi} below), 
we conclude $\dim (\Rc(K_{a,b}))=\dim (\kk[\Pi_{a,b}])=2(a+b-2)$. 
Moreover, by \cite[Corollary 3.3.5]{herzog-hibi-book}, it follows that if $\ini_<(\Rc(K_{a,b}))$ is Cohen-Macaulay/Gorenstein, then so is $\Rc(K_{a,b})$. 
The Gorensteinness of the Hibi ring of $\Pi$ is equivalent to that all of the maximal chains in $\Pi$ have the same length (\cite{hibiring}). 
\end{proof}

\subsection{Hibi rings of a finite poset}\label{sec:Hibi}

Next, we recall Hibi rings. 
Let $\Pi=\{p_1,\ldots,p_{d-1}\}$ be a finite poset equipped with a partial order $\prec$. 
For a subset $I \subset \Pi$, we say that $I$ is a \textit{poset ideal} if $I$ satisfies that ``$x \in I$ and $y \prec x$ imply $y \in I$''. 
We regard the empty set as a poset ideal. Let $\Ic(\Pi)$ be the set of all poset ideals of $\Pi$. 
Poset ideals in $\Pi$ are one-to-one corresponding to antichains in $\Pi$, 
i.e., a subset $A \subset \Pi$ satisfying $x \not\prec y$ and $y \not\prec x$ for any $x,y \in A$. 
In fact, for a given antichain $A$, we can associate a poset ideal \[\langle p : p \in A \rangle := \{q \in \Pi : q \preceq p \text{ for some }p \in A\} \subset \Pi.\] 
Conversely, given a poset ideal $I$, we may take the set of all maximal elements of $I$ with respect to $\prec$ to get the associated antichain. 

In what follows, we identify a subset $K \subset \Pi$ with a point in $\ZZ^\Pi (=\ZZ^{d-1})$ by $\sum_{p_i \in K}\eb_i$, 
where $\eb_i$ is the $i$-th unit vector of $\ZZ^{d-1}$ and $\emptyset$ corresponds to the origin.

\medskip

We introduce the $\kk$-algebra associated to $\Pi$ as follows (see \cite{hibiring}): 
\begin{align*}
\kk[\Pi]:=\kk[(\prod_{p_i \in I} x_i)x_d : I \in \Ic(\Pi)] \subset \kk[x_1,\ldots,x_d].
\end{align*}
We call $\kk[\Pi]$ the \textit{Hibi ring} of $\Pi$. 
The following facts are known: 
\begin{itemize}
\item $\dim (\kk[\Pi])=|\Pi|+1$; 
\item $\kk[\Pi]$ is a normal Cohen--Macaulay domain; 
\item $\kk[\Pi]$ is Gorenstein if and only if all of the maximal chains in $\Pi$ have the same length. 
\end{itemize}

Let $\kk[z_I : I \in \Ic(\Pi)]$ be a polynomial ring in $|\Ic(\Pi)|$ variables and consider the surjective ring homomorphism: 
\[\kk[z_I : I \in \Ic(\Pi)] \rightarrow \kk[\Pi]; \;\; z_I \mapsto (\prod_{p_i \in I} x_i)x_d.\]
We call the kernel of this map the \textit{toric ideal} of $\Pi$, denoted by $I_\Pi$. 
It is known that $I_\Pi$ can be described as follows: 
\begin{align}\label{pi}
I_\Pi=(z_Iz_J - z_{I \cup J}z_{I \cap J} : I,J \in \Ic(\Pi)).
\end{align}

\bigskip

\section{Proofs of Theorems~\ref{main1} and \ref{main2}}\label{sec:proofs}

In this section, we prove our main theorems (Theorems~\ref{main1} and \ref{main2}) at the same time. 

\subsection{Notation}\label{sec:label}
Let us fix our notation for our specific case $K_{a,b}$. 

Let $R_{a,b}$ and $S_{a,b}$ be the following polynomial rings in $(ab+\binom{a}{2}\binom{b}{2})$ and $2(a+b)$ variables, respectively: 
\begin{align*}
R_{a,b}&:=\kk[\{z_{ij} : 1 \leq i \leq a, \, 1 \leq j \leq b\} \cup \{z_{ii'j'j} : 1 \leq i<i' \leq a, \, 1 \leq j'<j \leq b\}]; \\
S_{a,b}&:=\kk[x_1,\ldots,x_a,x_1',\ldots,x_b',y_1',\ldots,y_a',y_1,\ldots,y_b]. 
\end{align*} 
We introduce a monomial order on $S_{a,b}$ as follows: 
\begin{equation}\label{order}
\begin{split}
&\text{$<$ is the graded lexicographic order with respect to the ordering of the variables} \\
&x_1>\cdots>x_a>x_1'>\cdots>x_b'>y_1'>\cdots>y_a'>y_1>\cdots>y_b.
\end{split}
\end{equation}
Note that a corresponding weight vector of $<$ can be chosen in the following manner: choose  
\[m_1 > \cdots > m_a > m_1' > \cdots > m_b' > n_1' > \cdots > n_a' > n_1 > \cdots > n_b > 0\]
satisfying e.g. $m_1>m_2+\cdots+n_b$, $m_2>m_3+\cdots+n_b$, and so on, 
and let \begin{align}\label{w}w_{a,b}:=(m_1,\ldots,n_b) \in \RR_{>0}^{2(a+b)}.\end{align}

Given $1 \leq i \leq a$ and $1 \leq j \leq b$, let 
\[f_{ij}:=x_iy_j-x_j'y_i'.\]
Moreover, given $1 \leq i <i' \leq a$ and $1 \leq j' <j \leq b$, let \begin{align*}
f_{ii'j'j}\,:=& \, f_{ij'}f_{i'j}-f_{ij}f_{i'j'} \\
=&\, x_ix_{j'}'y_{i'}'y_j+x_{i'}x_j'y_i'y_{j'}-x_ix_j'y_{i'}'y_{j'}-x_{i'}x_{j'}'y_i'y_j. 
\end{align*}
Let \[
\Gc_{a,b}=\{f_{ij} : 1 \leq i \leq a, \, 1 \leq j \leq b\} \cup \{f_{ii'j'j} : 1 \leq i<i' \leq a, \, 1 \leq j' < j \leq b\}. 
\]
Let $I$ be the kernel of the following ring homomorphism: 
\[\varphi_{a,b} : R_{a,b} \rightarrow S_{a,b}; \;\; \varphi_{a,b}(z_{ij})=f_{ij}, \;\; \varphi_{a,b}(z_{ii'j'j})=f_{ii'j'j}.\]
Namely, $\kk[\Gc_{a,b}]$ is isomorphic to $R_{a,b}/I$.

Let $A_{a,b} \subset (\ZZ_{\geq 0}^{a+b})^2$ be defined as follows: 
\begin{equation}\label{Aab}
\begin{split}
A_{a,b}:=&\{\eb_i+\fb_j : 1 \leq i \leq a, \, 1 \leq j \leq b\} \\
\cup &\{\eb_i+\eb_{i'}'+\fb_{j'}'+\fb_j : 1 \leq i<i' \leq a, \, 1 \leq j'<j \leq b\}, 
\end{split}
\end{equation}
where $\eb_1,\ldots,\eb_a,\fb_1,\ldots,\fb_b, \eb_1',\ldots,\eb_a',\fb_1',\ldots,\fb_b'$ form a $\ZZ$-basis of $\ZZ^{2(a+b)}$. Since 
\begin{align*}
\ini_<(f_{ij})=x_iy_j\;\text{ and }\; \ini_<(f_{ii'j'j})=x_ix_{j'}'y_{i'}'y_j 
\end{align*}
hold, we can find a one-to-one correspondence between $A_{a,b}$ and the initial monomials of $\Gc_{a,b}$ 
by assigning $x_i$ to $\eb_i$, $y_{i'}'$ to $\eb_{i'}'$, $x_{j'}'$ to $\fb_{j'}'$, and $y_j$ to $\fb_j$, respectively.  

\bigskip

Our goal of this section is to prove Theorems~\ref{main1} and \ref{main2}. 
To this end, we divide the discussions into the following four steps: 
\begin{itemize}
\item[(I)] Prove the unimodular equivalence of the associated affine monoid of $\ini_<(\Gc_{a,b})$ and the affine monoid corresponding to the poset ideals of $\Pi_{a,b}$; 
\item[(II)] Compute the toric ideal $I_{A_{a,b}}$ of $A_{a,b}$ by using (I); 
\item[(III)] Find (a part of) generators of $I$ and define $J$ generated by them;  
\item[(IV)] Prove $I_{A_{a,b}} \subseteq \ini_{A_{a,b}^Tw_{a,b}}(J)$. 
\end{itemize}
Then $I_{A_{a,b}} \subseteq \ini_{A_{a,b}^Tw_{a,b}}(J) \subseteq \ini_{A_{a,b}^Tw_{a,b}}(I) \subseteq I_{A_{a,b}}$ can be obtained. 
(The first inclusion is (IV), the second one follows by definition of $J$, and the third one comes from Theorem~\ref{sagbi}.) 
Hence, $I_{A_{a,b}} = \ini_{A_{a,b}^Tw_{a,b}}(J) = \ini_{A_{a,b}^Tw_{a,b}}(I)$, so we conclude that $\Gc_{a,b}$ forms a SAGBI basis with respect to $<$ by Theorem~\ref{sagbi}. 
We also obtain that $I=J$ simultaneously. 
Furthermore, since $\Gc_{a,b}$ forms a SAGBI basis for $\kk[\Gc_{a,b}]$, (I) implies that $\ini_<(\kk[\Gc_{a,b}]) \cong \kk[\Pi_{a,b}]$.

\subsection{Order ideals of $\Pi_{a,b}$ and the initial monomials of $\Gc_{a,b}$}\label{I}

The first step (I) is to construct a unimodular transformation from the initial monomials of $\Gc_{a,b}$ to the poset ideals of $\Pi_{a,b}$. 

Employing the notation as in Figure~\ref{piab}, we see that $\Ic(\Pi_{a,b})$ consists of the following poset ideals: 
\begin{equation}\label{Bab}
\begin{split}
&\emptyset, \; \langle e_i \rangle, \; \langle f_j \rangle, \; \langle e_i,f_j \rangle \; \text{ for } \; 1 \leq i \leq a-1, \; 2 \leq j \leq b \;\; \text{ and } \\
&\langle e_i, e_{i'}' \rangle, \; \langle e_i,e_{i'}',f_j \rangle, \; \langle f_{j'}',f_j \rangle, \; \langle e_i,f_{j'}',f_j \rangle, \; \langle e_i,e_{i'}',f_{j'}', f_j \rangle \\ 
&\quad\quad\quad\;\,\quad\quad\quad\quad\quad \text{ for } \; 1 \leq i<i' \leq a-1, \; 1 \leq j'<j \leq b. 
\end{split}
\end{equation}
We identify those poset ideals with those points in $\ZZ^{\Pi_{a,b}}$ whose $\ZZ$-basis is
\[\eb_1,\ldots,\eb_{a-1},\eb_2',\ldots,\eb_{a-1}'. \fb_1',\ldots,\fb_{b-1}',\fb_2,\ldots,\fb_b \in \ZZ^{\Pi_{a,b}} (=\ZZ^{2a+2b-5}).\]
More precisely, given $E \subset \{e_1,\ldots,e_{a-1}\}$, $E' \subset \{e_2',\ldots,e_{a-1}'\}$, $F' \subset \{f_1',\ldots,f_{b-1}'\}$ and $F \subset \{f_2,\ldots,f_b\}$, 
we associate $E \cup E' \cup F' \cup F \subset \Pi_{a,b}$ with 
\[\sum_{e_i \in E}\eb_{i}+\sum_{e_{i'}' \in E'}\eb_{i'}'+\sum_{f_{j'}' \in F'}\fb'_{j'}+\sum_{f_j \in F}\fb_j \in \ZZ^{\Pi_{a,b}}.\] 

Let $B_{a,b} \subset \ZZ^{\Pi_{a,b}}$ be the set of those points in $\ZZ^{\Pi_{a,b}}$ corresponding to \eqref{Bab}. 
Our goal of this subsection is to construct a unimodular transformation from $A_{a,b}$ to $B_{a,b}$. 

\medskip

Let us consider the following transformation. 
Given \[\sum_{i=1}^a c_i \eb_i + \sum_{i'=2}^a c_{i'}' \eb_{i'}' + \sum_{j'=1}^{b-1} d_{j'}' \fb_{j'}' + \sum_{j=1}^b d_j \fb_j \in A_{a,b}  
\text{ with }c_i,c_{i'}',d_{j'}',d_j \in \{0,1\},\] 
where $\eb_1'$ and $\fb_b'$ never appear by definition \eqref{Aab}, we apply
\begin{align*}
c_i &\mapsto c_1+\cdots+c_i \;\text{ for }1 \leq i \leq a, \\ 
c_i' &\mapsto c_2'+\cdots+c_{i'}' \text{ for }2 \leq i' \leq a-1, \;\; c_a' \mapsto c_2'+\cdots+c_a'-(d_1'+\cdots+d_{b-1}'), \\
d_{j'}' &\mapsto d_{j'}'+\cdots+d_b' \text{ for }1 \leq j' \leq b-1, \text{ and }\\
d_j &\mapsto d_j+\cdots+d_b \text{ for }1 \leq j \leq b.  
\end{align*}
Then the points in $A_{a,b}$ become to satisfy that each of the coefficients of $\eb_a$ and $\fb_1$ (resp. $\eb_a'$) is always $1$ (resp. $0$). 
By ignoring those three coordinates, the points in $A_{a,b}$ are unimodularly transformed as follows: 
\begin{align*}
\eb_a+\fb_1 &\mapsto {\bf 0} \overset{1 : 1}{\longleftrightarrow} \emptyset, \;\;
\eb_i+\fb_1 \mapsto \eb_i+\cdots+\eb_{a-1} \overset{1 : 1}{\longleftrightarrow} \langle e_i \rangle, \;\; 
\eb_a+\fb_j \mapsto \fb_2+\cdots+\fb_j \overset{1 : 1}{\longleftrightarrow} \langle f_j \rangle, \\
\eb_i+\fb_j &\mapsto \eb_i+\cdots+\eb_{a-1}+\fb_2+\cdots+\fb_j \overset{1 : 1}{\longleftrightarrow} \langle e_i,f_j \rangle 
\end{align*}
for $1 \leq i \leq a-1, \, 2 \leq j \leq b$,  
\begin{align*}
\eb_i+\eb_{i'}'+\fb_1'+\fb_2&\mapsto \eb_i+\cdots+\eb_{a-1}+\eb_{i'}'+\cdots+\eb_{a-1}'+\fb_1'+\fb_2 \overset{1 : 1}{\longleftrightarrow} \langle e_i,e_{i'}' \rangle, \\ 
\eb_i+\eb_{i'}'+\fb_1'+\fb_j&\mapsto \eb_i+\cdots+\eb_{a-1}+\eb_{i'}'+\cdots+\eb_{a-1}'+\fb_1'+\fb_2+\cdots+\fb_j 
\overset{1 : 1}{\longleftrightarrow} \langle e_i,e_{i'}',f_j \rangle
\end{align*}
for $1 \leq i<i' \leq a, \, 3 \leq j \leq b$, 
\begin{align*}
\eb_{a-1}+\eb_a'+\fb_{j'}'+\fb_j&\mapsto \eb_{a-1}+\fb_1'+\cdots+\fb_{j'}'+\fb_2+\cdots+\fb_j \overset{1 : 1}{\longleftrightarrow} \langle f_{j'}',f_j \rangle, \\ 
\eb_i+\eb_a'+\fb_{j'}'+\fb_j&\mapsto \eb_i+\cdots+\eb_{a-1}+\fb_1'+\cdots+\fb_{j'}'+\fb_2+\cdots+\fb_j \overset{1 : 1}{\longleftrightarrow} \langle e_i,f_{j'}',f_j \rangle, 
\end{align*}
for $1 \leq i \leq a-2, \, 1 \leq j'<j \leq b$, and 
\[\eb_i+\eb_{i'}'+\fb_{j'}'+\fb_j\mapsto \eb_i+\cdots+\eb_{a-1}+\eb_{i'}'+\cdots+\eb_{a-1}'+\fb_1'+\cdots+\fb_{j'}'+\fb_2+\cdots+\fb_j 
\overset{1 : 1}{\longleftrightarrow} \langle e_i,e_{i'}',f_{j'}',f_j \rangle\]
for $1 \leq i<i' \leq a-1, \, 2 \leq j'<j \leq b$.

\subsection{Computation of the toric ideal of $\Pi_{a,b}$}\label{II}

Let $A=A_{a,b}$. 
Consider the monomial algebra $\kk[A]$. Note that $\kk[A]=\kk[\ini_<(\Gc_{a,b})]$. 
Let $I_A$ be the kernel of the following surjective ring homomorphism: 
\begin{align*}
R_{a,b} \rightarrow \kk[A]; \quad\quad z_{ij} &\mapsto x_iy_j \quad\quad\, \text{ for } \; 1 \leq i \leq a, \, 1 \leq j \leq b, \\ 
z_{ii'j'j} &\mapsto x_iy_{i'}'x_{j'}'y_j \; \text{ for } \; 1 \leq i<i' \leq a, \, 1 \leq j'<j \leq b. 
\end{align*}
By Subsection~\ref{I}, $I_A$ is isomorphic to the toric ideal of $\Pi_{a,b}$. 
Our next goal (II) is to compute $I_A$ through this identification, i.e., by using the general description \eqref{pi} of the toric ideals of posets. 

Regarding the poset ideals of $\Pi_{a,b}$, we see that  
\begin{align}\label{eq:minmax}
\langle e_i,f_j \rangle \cup \langle e_p,f_q \rangle = \langle e_{\min\{i,p\}},f_{\max\{j,q\}} \rangle \text{ and }
\langle e_i,f_j \rangle \cap \langle e_p,f_q \rangle = \langle e_{\max\{i,p\}},f_{\min\{j,q\}} \rangle  
\end{align}
for $1 \leq i,p \leq a, 1 \leq j,q \leq b$. 
Other poset ideals are similar. By \eqref{eq:minmax}, \eqref{pi}, and the identification described in Subsection~\ref{I}, 
we see that $I_A$ (i.e. the toric ideal of $\Pi_{a,b}$) is generated by the following binomials: 
\begin{align*}
&z_{ij}z_{pq} - z_{iq}z_{pj} \;\,\text{ for }1 \leq i \leq p \leq a, \, 1 \leq j \leq q \leq b; \\
&z_{ij}z_{pp'q'q}-z_{\max\{i,p\}\min\{j,q\}}z_{\min\{i,p\}p'q'\max\{j,q\}} \\
&\quad\quad\quad\quad\quad\quad\text{for }1 \leq i \leq a, \, 1 \leq j \leq b, \, 1 \leq p<p' \leq a, \, 1 \leq q'<q \leq b; \\
&z_{ii'j'j}z_{pp'q'q}-z_{\min\{i,p\}\min\{i',p'\}\max\{j',q'\}\max\{j,q\}}z_{\max\{i,p\}\max\{i',p'\}\min\{j',q'\}\min\{j,q\}} \\
&\quad\quad\quad\quad\quad\quad\text{for } 1 \leq i<i' \leq a, \, 1 \leq j'<j \leq b, \, 1 \leq p<p' \leq a, \, 1 \leq q'<q \leq b. 
\end{align*}
On the first binomial, if $i=p$, then $z_{ij}z_{pq}-z_{iq}z_{pj}$ becomes $0$, so we can remove this case. Similarly, the case of $j=q$ is redundant. 
By similar discussions for other binomials, for $1 \leq i,i',p,p' \leq a$, $1 \leq j,j',q,q' \leq b$ with $i<i'$, $p<p'$, $j'<j$ and $q'<q$, 
we can reduce the above binomials as follows: 
\begin{equation}\label{eq:binom}
\begin{split}
&z_{ij}z_{pq} - z_{iq}z_{pj} \;\,\text{ for } i < p, \, j < q; \\
&z_{ij}z_{pp'q'q}-z_{pq}z_{ip'q'j} \;\,\text{ for }i < p, \, q < j; \\
&z_{ij}z_{pp'q'q}-z_{iq}z_{pp'q'j} \;\,\text{ for }p \leq i, \, q < j; \\
&z_{ij}z_{pp'q'q}-z_{pj}z_{ip'q'q} \;\,\text{ for }i < p, \, j \leq q; \\
&z_{ii'j'j}z_{pp'q'q}-z_{ii'q'q}z_{pp'j'j} \;\,\text{ for }i \leq p, \, i' \leq p', \, j' \leq q', \, j \leq q; \\
&z_{ii'j'j}z_{pp'q'q}-z_{pi'q'q}z_{ip'j'j} \;\,\text{ for }i > p, \, i' < p', \, j' \leq q', \, j \leq q; \\
&z_{ii'j'j}z_{pp'q'q}-z_{ip'q'q}z_{pi'j'j} \;\,\text{ for }i < p, \, i' > p', \, j' \leq q', \, j \leq q; \\
&z_{ii'j'j}z_{pp'q'q}-z_{ii'j'q}z_{pp'q'j} \;\,\text{ for }i \leq p, \, i' \leq p', \, j' > q', \, j < q; \\ 
&z_{ii'j'j}z_{pp'q'q}-z_{ii'q'j}z_{pp'j'q} \;\,\text{ for }i \leq p, \, i' \leq p', \, j' < q', \, j>q; \\ 
&z_{ii'j'j}z_{pp'q'q}-z_{pi'j'q}z_{ip'q'j} \;\,\text{ for }i > p, \, i' < p', \, j' > q', \, j < q \\
&\;\;\qquad\qquad\qquad\qquad\qquad\text{ or }i < p, \, i' > p', \, j' < q', \, j>q; \\
&z_{ii'j'j}z_{pp'q'q}-z_{pi'q'j}z_{ip'j'q} \;\,\text{ for }i > p, \, i' < p', \, j' < q', \, j > q \\
&\;\;\qquad\qquad\qquad\qquad\qquad\text{ or }i < p, \, i' > p', \, j' > q', \, j < q. 
\end{split}
\end{equation}

\subsection{Generators of $I$}\label{III}

Our next goal is to find some generators of $I$. 

For $1 \leq i,i',p,p' \leq a$, $1 \leq j,j',q,q' \leq b$ with $i<i'$, $p<p'$, $j'<j$ and $q'<q$, we introduce several polynomials in $R_{a,b}$:  
\begin{equation}\label{eq:J}
\begin{split}
&z_{ij}z_{pq}-z_{iq}z_{pj}-z_{ipjq}\;\, \text{ for }i<p, \, j<q; \\
&z_{ij}z_{pp'q'q}-z_{pq}z_{ip'q'j}+z_{iq'}z_{pp'qj}+z_{p'j}z_{ipq'q}+z_{p'q'}z_{ipqj} \;\,\text{ for }i<p, \, q<j; \\
&z_{ij}z_{pp'q'q}-z_{iq}z_{pp'q'j}+z_{iq'}z_{pp'qj} \;\,\text{ for }p \leq i, \, q<j; \\
&z_{ij}z_{pp'q'q}-z_{pj}z_{ip'q'q}+z_{p'j}z_{ipq'q} \;\,\text{ for }i < p, \, j \leq q; \\
&z_{ii'j'j}z_{pp'q'q}-z_{ii'q'q}z_{pp'j'j} \;\,\text{ for }i \leq p, \, i' \leq p', \, j' \leq q', \, j \leq q; \\
&z_{ii'j'j}z_{pp'q'q}-z_{pi'q'q}z_{ip'j'j}+z_{piq'q}z_{i'p'j'j} \;\,\text{ for }i > p, \, i' < p', \, j' \leq q', \, j \leq q; \\
&z_{ii'j'j}z_{pp'q'q}-z_{ip'q'q}z_{pi'j'j}+z_{p'i'j'j}z_{ipq'q} \;\,\text{ for }i < p, \, i' > p', \, j' \leq q', \, j \leq q; \\
&z_{ii'j'j}z_{pp'q'q}-z_{ii'j'q}z_{pp'q'j}+z_{ii'jq}z_{pp'q'j'} \;\,\text{ for }i \leq p, \, i' \leq p', \, j' > q', \, j < q; \\ 
&z_{ii'j'j}z_{pp'q'q}-z_{ii'q'j}z_{pp'j'q}+z_{ii'j'q'}z_{pp'qj} \;\,\text{ for }i \leq p, \, i' \leq p', \, j' < q', \, j>q; \\ 
&z_{ii'j'j}z_{pp'q'q}-z_{pi'j'q}z_{ip'q'j}+z_{pij'q}z_{i'p'q'j}+z_{ip'jq}z_{pi'q'j'}-z_{i'p'jq}z_{piq'j'} \\
&\quad\quad\quad\;\;\;\quad\quad\quad\quad\quad\quad\quad\quad\quad\quad\quad\quad\text{ for }i > p, \, i' < p', \, j' > q', \, j < q; \\
&z_{ii'j'j}z_{pp'q'q}-z_{pi'j'q}z_{ip'q'j}+z_{ipq'j}z_{p'i'j'q}+z_{pi'qj}z_{ip'j'q'}-z_{p'i'qj}z_{ipj'q'} \\
&\quad\quad\quad\;\;\;\quad\quad\quad\quad\quad\quad\quad\quad\quad\quad\quad\quad\text{ for }i < p, \, i' > p', \, j' < q', \, j>q; \\
&z_{ii'j'j}z_{pp'q'q}-z_{pi'q'j}z_{ip'j'q}+z_{piq'j}z_{i'p'j'q}+z_{ip'qj}z_{pi'j'q'}-z_{i'p'qj}z_{pij'q'} \\
&\quad\quad\quad\;\;\;\quad\quad\quad\quad\quad\quad\quad\quad\quad\quad\quad\quad\text{ for }i > p, \, i' < p', \, j' < q', \, j > q; \\
&z_{ii'j'j}z_{pp'q'q}-z_{pi'q'j}z_{ip'j'q}+z_{ipj'q}z_{p'i'q'j}+z_{pi'jq}z_{ip'q'j'}-z_{p'i'jq}z_{ipq'j'} \\
&\quad\quad\quad\;\;\;\quad\quad\quad\quad\quad\quad\quad\quad\quad\quad\quad\quad\text{ for }i < p, \, i' > p', \, j' > q', \, j < q. 
\end{split}
\end{equation}
Then it is straightforward to check that these polynomials belong to $I=\mathrm{Ker} (\varphi_{a,b})$.  

Let $J$ be the ideal generated by the polynomials \eqref{eq:J}. 
Clearly, $J \subseteq I$ holds.  

\subsection{Proof of Theorems~\ref{main1} and \ref{main2}}

Let $w=w_{a,b}$ be the one given in \eqref{w}. 

Our final task is to show that $I_A \subseteq \ini_{A^Tw}(J)$. 
Since $I_A$ is generated by the binomials described in \eqref{eq:binom}, 
it is enough to show that for any binomial $f$ in \eqref{eq:binom}, we can find a polynomial $g$ from \eqref{eq:J} such that $f=\ini_{A^Tw}(g)$. 

Here, by definition of $w$, the weights of the variables $z_{ij}$ and $z_{ii'j'j}$ are as follows: 
\[(\text{weight of }z_{ij})=m_i+n_j \;\;\text{ and }\;\;(\text{weight of }z_{ii'j'j})=m_i+m_{i'}'+n_{j'}'+n_j.\]
By taking these into account, for example, the weight of the first two monomials of $z_{ij}z_{pq}-z_{iq}z_{pj}-z_{ipjq}$ is $m_i+m_p+n_j+n_q$, 
while that of the third one is $m_i+m_p'+n_j'+n_q$. 
Since $m_p > m_p'$, we obtain that \[\ini_{A^Tw}(z_{ij}z_{pq}-z_{iq}z_{pj}-z_{ipjq})=z_{ij}z_{pq}-z_{iq}z_{pj},\] 
which corresponds to the first binomial of \eqref{eq:binom}. 

Similarly, we can verify that the first two monomials form the initial form $\ini_{A^Tw}(g)$ for each $g$ in \eqref{eq:J}, 
and each of those two monomials corresponds to the binomial in \eqref{eq:binom}, respectively.  
These discussions imply that $I_A \subseteq \ini_{A^Tw}(J)$, as required. 

\bigskip

\section{Further questions}\label{sec:q}

In this section, we study some further problems on $\Rc(G)$ for general graphs $G$.

One of the most fundamental invariants on commutative rings is its Krull dimension. 
Regarding the Krull dimension of $\Rc(G)$, we have the following: 
\begin{prop}\label{prop:dim}
Let $G$ be a graph on $d$ vertices with $n$ edges. 
\begin{itemize}
\item[(i)] If $G$ is bipartite, then \[\dim(\Rc(G)) \leq \min\{n,2d-4\}. \]
\item[(ii)] If $G$ is non-bipartite, then \[\dim(\Rc(G)) \leq \min\{n,2d-3\}. \]
\end{itemize}
\end{prop}
\begin{proof}
In general, $G$ is always a subgraph of $K_d$. This means that $\Rc(G)$ is a subring of $\Rc(K_d)=\Ac_{2,d}$. 
It is well known that $\dim (\Ac_{k,d})=k(d-k)+1$. 
(This can be seen, e.g., via that $\dim (\Ac_{k,d})$ coincides with the Krull dimension of its initial algebra, 
which is equal to $k(d-k)+1$. See \cite[Chapter 11]{sturmfels} for more details.) 
Hence, $\dim (\Rc(G)) \leq 2(d-2)+1=2d-3$ follows. 

In the case where $G$ is bipartite, there is a positive integer $a$ such that $G$ is a subgraph of $K_{a,d-a}$, i.e., $\Rc(G)$ is a subring of $\Rc(K_{a,d-a})$. 
Since $\dim (\Rc(K_{a,d-a}))=2d-4$ by Corollary~\ref{kei}, we obtain $\dim (\Rc(G)) \leq 2d-4$. 

Finally, $\dim(\Rc(G)) \leq n$ follows from that $\Rc(G)$ is generated by $n$ polynomials. 
\end{proof}

Now, it is natural to think of when the equality holds. Namely, 
\begin{q}
Let $G$ be a graph on $d$ vertices with $n$ edges. 
When does $\dim(\Rc(G))=\min\{n,2d-3\}$ hold? Or when does $\dim(\Rc(G))=\min\{n,2d-4\}$ hold if $G$ is bipartite? 
\end{q}

\begin{ex}\label{ex:tree}
(i) Let $G$ be a tree with $d$ vertices. Then $\Rc(G)$ is isomorphic to the polynomial ring $\kk[z_1,\ldots,z_{d-1}]$. In particular, $\dim(\Rc(G))=d-1$. 
In fact, the points in $\ZZ^{2d}$ corresponding to the initial monomials of $f_{ij}$ for $\{i,j\} \in E(G)$ are linearly independent. 
This implies that $\dim(\ini_<(\Rc(G))) \geq d-1$. Hence, $\dim(\Rc(G)) \geq d-1$, while we know $\dim(\Rc(G)) \leq d-1$ by Proposition~\ref{prop:dim} (ii). 

\noindent
(ii) By a similar discussion to (i), we can also verify that $\Rc(G) \cong \kk[z_1,\ldots,z_d]$ if $G$ is unicyclic, 
i.e., $G$ is connected and contains exactly one cycle. 
In the case where $G$ contains an odd cycle, the points corresponding to the initial monomials for $d$ edges of $G$ become linearly independent. 
In the case where $G$ contains an even cycle, those points are not linearly independent, but we can find an additional initial monomial of $\Rc(G)$. 
In fact, let $G$ be an even cycle with $E(G)=\{\{i,i+1\} : i=1,\ldots,2k-1\} \cup \{\{1,2k\}\}$. (The general case of unicyclic graphs is similar.) 
Then the points corresponding to $\ini_<(f_{12}),\ldots,\ini_<(f_{(2k-1) \, 2k})$ and 
$\ini_<(\prod_{j=1}^kf_{(2j-1) \, 2j} - f_{1 \, 2k}\prod_{j=1}^{k-1} f_{2j \, (2j+1)})$ are linearly independent. 
\end{ex}

\begin{ex}\label{ex:other}
(i) Let $G_1$ be a graph on $8$ vertices with the edge set 
{\small\[E(G_1)=\{\{1, 2\}, \{2, 3\}, \{3, 4\}, \{4, 5\}, \{5, 6\}, \{6, 7\}, \{7, 8\}, \{1, 8\}, \{1, 4\}, \{1, 5\}, \{4, 8\}, \{5, 8\}\}.\]} See Figure~\ref{g1}.
\begin{figure}[h]
\begin{tikzpicture}[scale=0.8	]
    \node[circle, draw, fill=white, inner sep=2pt] (1) at (90:2) {1};
    \node[circle, draw, fill=white, inner sep=2pt] (2) at (45:2) {2};
    \node[circle, draw, fill=white, inner sep=2pt] (3) at (0:2) {3};
    \node[circle, draw, fill=white, inner sep=2pt] (4) at (315:2) {4};
    \node[circle, draw, fill=white, inner sep=2pt] (5) at (270:2) {5};
    \node[circle, draw, fill=white, inner sep=2pt] (6) at (225:2) {6};
    \node[circle, draw, fill=white, inner sep=2pt] (7) at (180:2) {7};
    \node[circle, draw, fill=white, inner sep=2pt] (8) at (135:2) {8};
    
    \draw (1) -- (2);
    \draw (2) -- (3);
    \draw (3) -- (4);
    \draw (4) -- (5);
    \draw (5) -- (6);
    \draw (6) -- (7);
    \draw (7) -- (8);
    \draw (8) -- (1);
    \draw (1) -- (4);
    \draw (1) -- (5);
    \draw (4) -- (8);
    \draw (5) -- (8);
\end{tikzpicture}
\caption{Graph $G_1$}\label{g1}
\end{figure}
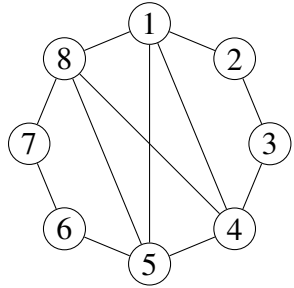
Then we see that $\dim (\Rc(G_1))=11$. This example says that the equality of Proposition~\ref{prop:dim} (i) does not necessarily hold. 

\noindent
(ii) Let $G_2$ be a bipartite graph on $8$ vertices with the following $12$ edges: 
{\small\[E(G_2)=\{\{1, 2\}, \{2, 3\}, \{3, 4\}, \{4, 5\}, \{5, 6\}, \{1, 4\}, \{1, 6\}, \{2, 5\}, \{3, 6\}, \{6, 7\}, \{7, 8\}, \{5, 8\}\}.\]} See Figure~\ref{g2}. 
\begin{figure}[h]
\begin{tikzpicture}[scale=0.8	]
    \node[circle, draw, fill=white, inner sep=2pt] (1) at (0,2) {1};
    \node[circle, draw, fill=white, inner sep=2pt] (3) at (2,2) {3};
    \node[circle, draw, fill=white, inner sep=2pt] (5) at (4,2) {5};
    \node[circle, draw, fill=white, inner sep=2pt] (7) at (6,2) {7};
    \node[circle, draw, fill=white, inner sep=2pt] (2) at (0,0) {2};
    \node[circle, draw, fill=white, inner sep=2pt] (4) at (2,0) {4};
    \node[circle, draw, fill=white, inner sep=2pt] (6) at (4,0) {6};
    \node[circle, draw, fill=white, inner sep=2pt] (8) at (6,0) {8};
    
    \draw (1) -- (2);
    \draw (2) -- (3);
    \draw (3) -- (4);
    \draw (4) -- (5);
    \draw (5) -- (6);
    \draw (1) -- (4);
    \draw (1) -- (6);
    \draw (2) -- (5);
    \draw (3) -- (6);
    \draw (6) -- (7);
    \draw (7) -- (8);
    \draw (5) -- (8);
\end{tikzpicture}
\caption{Graph $G_2$}\label{g2}
\end{figure}
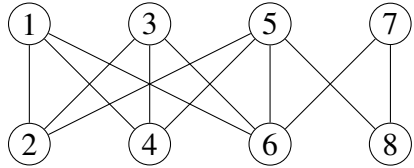
Then we see that $\dim (\Rc(G_2))=11$. This example says that the equality of Proposition~\ref{prop:dim} (ii) does not necessarily hold, either. 

\medskip

These computations are performed by using {\tt Macaulay2}. 
\end{ex}

Finally, we mention the Cohen-Macaulayness of $\Rc(G)$. 
As far as the author computes, $\Rc(G)$ is always Cohen-Macaulay. Thus, the following problem naturally arises: 
\begin{prob}
Prove that $\Rc(G)$ is Cohen-Macaulay for any $G$, and characterize when $\Rc(G)$ is Gorenstein. 
\end{prob}

\bigskip

\bibliographystyle{amsalpha}
\bibliography{bibliography}

\end{document}